\documentclass[11pt,reqno]{amsart}

\usepackage{amsfonts,amscd,latexsym,amsmath,amssymb,enumerate,verbatim,amsthm,epsfig,alltt,cancel,hyperref}
\theoremstyle{plain}
\newtheorem{master}{Master}[section]
\newtheorem{prop}[master]{Proposition}
\newtheorem{thm}[master]{Theorem}

\newtheorem{lem}[master]{Lemma}
\newtheorem{cor}[master]{Corollary}

\newtheorem{question}[master]{Question}

\newtheorem{Def}[master]{Definition}

\newcommand{\Com}{\mathbb{C}}
\newcommand{\Hil}{\mathcal{H}}

\newcommand{\re}{\mathrm{Re}}

\newcommand{\Rep}{\mathrm{Rep}}

\newcommand \NN{\mathbb{N}}
\newcommand \QQ{\mathbb{Q}}
\newcommand \CC{\mathbb{C}}

\theoremstyle{definition}

\theoremstyle{remark}
\newtheorem{remark}[master]{Remark}
\numberwithin{equation}{section}

\title[Property (T) and unitary representations]{Property (T), finite-dimensional representations, and generic representations}
\author{Michal Doucha}
\address{Institute of Mathematics, Czech Academy of Sciences, \v Zitn\' a 25, 115 67 Praha 1, Czech republic}
\email{doucha@math.cas.cz}
\author{Maciej Malicki}
\address{Department of Mathematics and Mathematical Economics, Warsaw School of Economics, al. Niepodleg\l o\' sci 162, 02-554 Warsaw, Poland}
\email{mamalicki@gmail.com}
\author{Alain Valette}
\address{Institut de Math\'ematiques, Universit\'e de Neuch\^atel, Unimail, 11 Rue Emile Argand, 2000 Neuch\^atel, Switzerland}
\email{alain.valette@unine.ch}
\begin{document}

\begin{abstract} Let $G$ be a discrete group with property (T). It is a standard fact that, in a unitary representation of $G$ on a Hilbert space $\Hil$, almost invariant vectors are close to invariant vectors, in a quantitative way. We begin by showing that, if a unitary representation has some vector whose coefficient function is close to a coefficient function of some finite-dimensional unitary representation $\sigma$, then the vector is close to a sub-representation isomorphic to $\sigma$: this makes quantitative a result of P.S. Wang \cite{Wa}. We use that to give a new proof of a result by D. Kerr, H. Li and M. Pichot \cite{KLP}, that a group $G$ with property (T) and such that $C^*(G)$ is residually finite-dimensional, admits a unitary representation which is generic (i.e. the orbit of this representation in $\Rep(G,\Hil)$ under the unitary group $U(\Hil)$ is comeager). We also show that, under the same assumptions, the set of representations equivalent to a Koopman representation is comeager in $\Rep(G,\Hil)$.
\end{abstract}

\maketitle

\section{Introduction}

Let $G$ be a discrete group and $\pi$ be a unitary representation of $G$ on some Hilbert space $\Hil$. For a finite set $F\subset G$ and $\varepsilon>0$, a vector $\xi\in\Hil$ is $(F,\varepsilon)$-invariant if $\max_{g\in F}\|\pi(g)\xi-\xi\|<\varepsilon$. Recall that $\pi$ almost has invariant vectors if, for every pair $(F,\varepsilon)$, the group $G$ has $(F,\varepsilon)$-vectors; and that the group $G$ has \emph{Kazhdan's property (T)} or is a \emph{Kazhdan group} if every unitary representation of $G$ almost having invariant vectors, has non-zero invariant vectors; see e.g. \cite{BdlHV} for Property (T). The definition can be reformulated in terms of weak containment of representations: $G$ has Property (T) if every unitary representation weakly containing the trivial representation of $G$, contains it strongly (see Remark 1.1.2 in \cite{BdlHV}). Crucial for us is an equivalent characterization due to P.S. Wang (Corollary 1.9 and Theorem 2.1 in \cite{Wa}): the group $G$ has property (T) if and only if for some (hence every) irreducible finite-dimensional unitary representation $\sigma$ of $G$, every unitary representation $\pi$ of $G$ that contains $\sigma$ weakly, contains it strongly.

It is a simple but useful fact that, if $G$ has property (T) and $\pi$ is a unitary representation almost having invariant vectors, ``almost invariant vectors are close to invariant vectors''. More precisely:

\begin{prop}[Proposition 1.1.9 in \cite{BdlHV}]\label{PropBHV} Let $G$ be a Kazhdan group. If $S$ is a finite generating set of $G$ and $\varepsilon_0$ is the corresponding Kazhdan constant, then for every $\delta\in ]0,1[$ and every unitary representation $\pi$ of $G$, any $(S,\varepsilon_0\delta)$-invariant vector $\xi$ satisfies $\|\xi-P\xi\|\leq\delta\|\xi\|$, where $P$ is the orthogonal projection on the subspace of $\pi(G)$-invariant vectors.
\hfill$\square$
\end{prop}

For a Kazhdan group $G$ and a unitary representation $\pi$ of $G$, fix a unit vector $\xi$ and look at the coefficient function
$$\phi_{\pi,\xi}(g)=\langle\pi(g)\xi,\xi\rangle\;(g\in G).$$
The question we first address in this paper is: if $\phi_{\pi,\xi}$ is close to some coefficient of an irreducible finite-dimensional unitary representation $\sigma$ of $G$, must $\xi$ be close to a finite-dimensional invariant subspace of $\pi$ carrying a sub-representation isomorphic to $\sigma$? We will see that, in analogy to Proposition \ref{PropBHV}, the answer is positive - with some care. 

\begin{Def}\label{good} Let $G$ be a finitely generated group with a symmetric finite generating set $S\subseteq G$ and let $\phi$ be some normalized positive definite function on $G$ associated with a unitary irreducible representation $\sigma$, of finite dimension $d$. Let $\pi$ be some unitary representation of $G$ on $\Hil$. Let $\varepsilon>0$. Say that a unit vector $\xi\in\Hil$ is $(\pi,\phi,\varepsilon)$-good if for every $s\in S^{2d^2+1}$ we have $|\phi_{\pi,\xi}(s)-\phi(s)|<\varepsilon$.
\end{Def}

Note that $S^k$ is just the ball of radius $k$ centered at the identity in $G$. So there is a certain lack of uniformity in Definition \ref{good}: we require an approximation of $\phi_{\pi,\xi}$ by $\phi$ on a ball whose size depends on the dimension of the representation $d$. Our main result, proved in section 2, can be viewed as a quantitative version of Wang's result.

\begin{thm}\label{mainresult}
Let $G$ be a discrete Kazhdan group, $S$ a finite symmetric generating set with $e\in S$, and let $\phi$ be a normalized positive definite function on $G$ associated with a finite-dimensional unitary irreducible representation $\sigma$ of $G$. For every $0<\delta<1$ there exists $\varepsilon_{\phi, \delta}>0$ such that for every unitary representation $\pi$ of $G$ on a Hilbert space $\Hil$, and every unit vector $x \in \Hil$ that is $(\pi,\phi,\varepsilon_{\phi, \delta})$-good, there exists a unit vector $x' \in \Hil$ with $\|x-x' \|\leq\delta$ such that the restriction of $\pi$ to the span of $\pi(G)x'$ is isomorphic to $\sigma$.
\end{thm}

In section 3, we apply Theorem \ref{mainresult} to the study of the global structure of the space of unitary representations of Kazhdan groups. Let us start with the notation. Let $G$ be an arbitary countable group and let $\Hil$ be a separable infinite-dimensional Hilbert space. The set $\Rep(G,\Hil)$ of all homomorphisms from $G$ into the unitary group $U(\Hil)$ can be viewed as a closed subset of the product space $U(\Hil)^G$, when we equip $U(\Hil)$ with the strong operator topology. With this identification, $\Rep(G,\Hil)$ is a Polish (i.e. separable and completely metrizable) space. We refer the reader to the monograph \cite{Ke2}, especially to the section on the spaces of unitary representations, for more information about this point of view on unitary representations. Recall that two unitary representations $\pi_1,\pi_2\in\Rep(G,\Hil)$ are \emph{isomorphic}, or \emph{unitarily equivalent} if there is a unitary operator $\phi\in U(\Hil)$ such that $\pi_1(g)=\phi\pi_2(g)\phi^*$, for every $g\in G$. Notice that this is an orbit equivalence relation given by the action of the unitary group $U(\Hil)$ on the space $\Rep(G,\Hil)$ by conjugation. Kechris raised a question (see again the section on the space of unitary representations in \cite{Ke2}) if there are countable groups with a \emph{generic unitary representation}, where ``generic'' here means its conjugacy class is large in the sense of Baire category, i.e. a representation whose class under the unitary equivalence contains a dense $G_\delta$ subset. As a matter of fact, we mention that it follows from the topological zero-one law that for every countable group $G$ either there is a generic representation in $\Rep(G,\Hil)$, or all conjugacy classes are meager (see e.g. Theorem 8.46 in \cite{Ke1}; to apply it, note that there is a dense conjugacy class in $\Rep(G,\Hil)$ --- indeed, take some countable dense set of representations from $\Rep(G,\Hil)$ and consider their direct sum).

Here as an application of Theorem \ref{mainresult} we prove the following result.

\begin{thm}\label{thm:generic_rep}
Let $G$ be a discrete Kazhdan group such that finite-dimensional representations are dense in the unitary dual $\hat{G}$. Then there is a generic unitary representation of $G$.
\end{thm}

We note that, although not explicitly stated there, this result already follows from a more general result of Kerr, Li and Pichot from \cite{KLP}, where they prove (see Theorem 2.5 there) that if $A$ is a separable C*-algebra where finite-dimensional representations are dense in $\hat{A}$, then there is a dense $G_\delta$ class in $\Rep(A,\Hil)$. Theorem \ref{thm:generic_rep} is then a special case for $A=C^*(G)$. Our proof is nevertheless done by more elementary means, in particular it does not invoke Voiculescu's theorem (see the proof of Theorem 2.5 in \cite{KLP} for details).\\

Another open question posed by Kechris as Problem H.16 in \cite{Ke2} is whether the subset of those representations $\pi\in\Rep(G,\Hil)$, where $G$ is still a countable group, that are equivalent to Koopman representations is meager in $\Rep(G,\Hil)$. Such representations are called \emph{realizable by an action} in \cite{Ke2}. Let us recall the terminology first. Let $(X,\mu)$ be a standard probability space (i.e. a space isomorphic to the unit interval $[0,1]$ equipped with the Lebesgue measure). Let $\alpha: G\curvearrowright (X,\mu)$ be an action of a countable group $G$ on $X$ by measure preserving measurable transformations.  Consider the unitary representation $\pi_\alpha: G\rightarrow L^2(X,\mu)$ defined by $\pi_\alpha(g)f(x)=f(\alpha(g^{-1},x))$, for every $f\in L^2(X,\mu)$. The \emph{Koopman representation} of $\alpha$ is the restriction of $\pi_\alpha$ to the invariant subspace $L^2_0(X,\mu)$, which is the orthogonal complement of the invariant subspace of constant functions.

In section 4 we prove the following result addressing the question of Kechris.

\begin{thm}\label{kechris}
Let $G$ be a discrete Kazhdan group such that finite-dimensional representations are dense in the unitary dual $\hat{G}$. Then the set of representations realizable by an action is comeager in $\Rep(G,\Hil)$.
\end{thm}

Let us mention that the condition that finite-dimensional representations are dense in the unitary dual $\hat{G}$ is, by the result of Archbold from \cite{Ar}, equivalent with the statement that the full group C*-algebra $C^*(G)$ is residually finite-dimensional. That is in turn, by the result of Exel and Loring from \cite{EL} (see also \cite{PeUs}), equivalent with the statement that finite-dimensional representations are dense in $\Rep(G,\Hil)$, which we shall use in the proof. Note that we call a representation $\pi\in\Rep(G,\Hil)$ finite-dimensional if the subalgebra $\pi(G)$ generates in $B(\Hil)$ is finite-dimensional.

The existence of infinite discrete Kazhdan groups with residually finite-dimensional C*-algebras seems to be open --- see Question 7.10 in \cite{BdlHV} and also Question 6.5 of Lubotzky and Shalom in \cite{LuSh} where they ask if there are infinite discrete Kazhdan groups with property FD, which is strictly stronger than having a residually finite-dimensional C*-algebra (a group has \emph{property FD} if representations factoring through finite groups are dense in the unitary dual).

\begin{question} It is known that being residually finite is not a sufficient condition to have a residually finite-dimensional C*-algebra by a result of Bekka \cite{Be}. However how about being LERF? (Recall that a finitely generated group is LERF if any finitely generated subgroup is the intersection of the finite index subgroups containing it). Ershov and Jaikin-Zapirain constructed in \cite{ErJZ} a Kazhdan group which is LERF. Is its group C*-algebra residually finite-dimensional?
\end{question}
\begin{remark}
We note that on the other hand we cannot exclude that it is possible to prove by a different argument that for every infinite group $G$, all classes in $\Rep(G,\Hil)$ are meager. That would together with Theorem \ref{thm:generic_rep} give that there are no infinite Kazhdan groups with a residually finite-dimensional C*-algebra.
\end{remark}
\noindent{\bf Acknowledgements: }The first named author was supported by the GA\v CR project 16-34860L and RVO: 67985840.

\section{A quantitative version of Wang's theorem}

Let $G$ be an infinite, finitely generated group. Let $S$ be a finite, symmetric, generating set of $G$, with $e\in S$. Let $\mathbb{C}G$ be the complex group ring of $G$. 

\subsection{Quantifying the Burnside theorem}

Let $\sigma$ be an irreducible unitary representation of dimension $d$, i.e. a homomorphism $\sigma:G \rightarrow U_d(\mathbb{C})$ such that $\sigma(G)$ has no proper invariant subspace. The classical Burnside theorem says that $\sigma(\mathbb{C}G)=M_d(\mathbb{C})$, i.e. $\sigma(G)$ contains a basis of $M_d(\mathbb{C})$. 

\begin{Def} Set $k(\sigma)=\min\{k>0: \dim_\mathbb{C}span \,\sigma(S^k)=d^2\}$.
\end{Def}

\begin{lem}\label{burnside} There is a constant $C>0$ (only depending on $S$) such that $C\log d\leq k(\sigma)\leq d^2$.
\end{lem}

\begin{proof} We have
$$d^2= \dim_\mathbb{C} span\, \sigma(S^{k(\sigma)})\leq |\sigma(S^{k(\sigma)})|\leq |S^{k(\sigma)}|\leq |S|^{k(\sigma)}.$$
Taking logarithms: $\frac{2}{\log |S|}\log d\leq k(\sigma)$.
To prove the upper bound, observe that the sequence $span\,\sigma(S^k)$ of subspaces of $M_d(\mathbb{C})$, is strictly increasing for $k<k(\sigma)$. Indeed, assume that $k$ is such that $span \,\sigma(S^k)= span\,\sigma(S^{k+1})$: this means that $span\,\sigma(S^k)$ is stable by left multiplication by $\sigma(S)$, hence by $\sigma(G)$ as $S$ is generating. Since the identity matrix is in $\sigma(S^k)$, we have $\sigma(G)\subset span\,\sigma(S^k)$, hence $k\geq k(\sigma)$. From this it is clear that $k(\sigma)\leq d^2$.
\end{proof}

\medskip
Let $v$ be a unit vector in $\mathbb{C}^d$. Since $v$ is cyclic for $\sigma(G)$, the map: 
$$T_v:\mathbb{C}S^{k(\sigma)}\rightarrow\mathbb{C}^d:f\mapsto\sigma(f)v$$
 is onto. Let $(\ker T_v)^\perp$ denote the orthogonal of $\ker T_v$ in $\mathbb{C}S^{k(\sigma)}$, let $U_v$ be the inverse of the map $T_v|_{(\ker T_v)^\perp}$. Endow $\mathbb{C}S^{k(\sigma)}$ with the $\ell^1$-norm, and let $\|U_v\|_{2\rightarrow 1}$ be the corresponding operator norm of $U_v$. So for every $w$ a unit vector in $\mathbb{C}^d$, there exists $f\in \mathbb{C}S^{k(\sigma)}$ with $\|f\|_1\leq \|U_v\|_{2\rightarrow 1}$, such that $\sigma(f)v=w$. 

\begin{lem}\label{compact} There exists $M>0$ such that for every two unit vectors $v, w\in\mathbb{C}^d$, there exists $f\in \mathbb{C}S^{k(\sigma)}$ with $\|f\|_1\leq M$, such that $\sigma(f)v=w$.
\end{lem}

\begin{proof} This is the preceding observation plus compactness of the unit sphere in $\mathbb{C}^d$: the constant $M=\max_{\|v\|=1}\|U_v\|_{2\rightarrow 1}$ does the job.
\end{proof}

\subsection{From weak containment to weak containment \`a la Zimmer}\label{weakcon}

Recall that, if $\pi,\rho$ are unitary representations of a discrete group $G$, the representation $\pi$ is weakly contained in the representation $\rho$ (i.e. $\pi\preceq\rho$) if every function of positive type associated with $\pi$ can be pointwise approximated by finite sums of positive definite type associated with $\rho$. If $\pi$ is irreducible, this is equivalent to require that every normalized function of positive type associated with $\pi$ can be pointwise approximated by normalized functions of positive type associated with $\rho$ (see Proposition F.1.4 in \cite{BdlHV}).

Zimmer introduced in Definition 7.3.5 of \cite{Zim} a different, inequivalent notion of weak containment. A $n\times n$-submatrix of $\pi$ is a function 
$$G\rightarrow M_n(\Com):g\mapsto (\langle \pi(g)e_i,e_j\rangle)_{1\leq i,j\leq n}$$
where $\{e_1,...,e_n\}$ is an orthonormal family in $\Hil_\pi$. Say that $\pi$ is weakly contained in $\rho$ in Zimmer's sense (i.e. $\pi\preceq_Z \rho$) if, for every $n>0$, every $n\times n$-submatrix of $\pi$ can be pointwise approximated by $n\times n$-submatrices of $\rho$. The exactly relation with the classical notion recalled above, is worked out in Remark F.1.2(ix) in \cite{BdlHV}; in particular, when $\pi$ is irreducible, $\pi\preceq\rho$ implies $\pi\preceq_Z \rho$. Our first goal will be to make the latter statement quantitative. For this we need a definition.

Let $\phi$ be associated with $\sigma$, as in Definition \ref{good}. Let $v$ be a unit vector in $\Hil_\sigma$ such that $\phi=\phi_{\sigma,v}$. Let $e_1,...,e_d$ be an orthonormal basis of $\mathbb{C}^d$; by lemma \ref{compact}, we find functions $f_1,...,f_d\in\mathbb{C}S^{k(\sigma)}$, with $\max_i \|f_i\|_1\leq M$, such that $\sigma(f_i)v=e_i\,(i=1,...,d)$.

\begin{lem}\label{Zimmer} Let $\pi\in \Rep(G,\Hil)$ be a unitary representation. Assume there is $\varepsilon>0$ and a unit vector $\eta\in\mathcal{H}$ such that for $s\in S^{2k(\sigma)+1}$ we have $|\langle\pi(s)\eta,\eta\rangle-\langle\sigma(s)v,v\rangle|<\varepsilon$. Set $\eta_i=\pi(f_i)\eta$. Then for $i,j=1,...,d$ and $g\in S$:
$$|\langle\sigma(g)e_i,e_j\rangle-\langle\pi(g)\eta_i,\eta_j\rangle|\leq\varepsilon M^2.$$
\end{lem}

\begin{proof} For $g\in S$:
$$|\langle\sigma(g)e_i,e_j\rangle-\langle\pi(g)\eta_i,\eta_j\rangle|=|\langle\sigma(g)\sigma(f_i)v,\sigma(f_j)v\rangle-\langle\pi(g)\pi(f_i)\eta,\pi(f_j)\eta\rangle|$$
$$=|\sum_{s,t\in G} f_i(s)\overline{f_j(t)}(\langle\sigma(t^{-1}gs)v,v\rangle - \langle\pi(t^{-1}gs)\eta,\eta\rangle)|$$
$$\leq\sum_{s,t\in G}|f_i(s)||f_j(s)||\langle\sigma(t^{-1}gs)v,v\rangle - \langle\pi(t^{-1}gs)\eta,\eta\rangle|.$$
Since the supports of the $f_i$'s are contained in $S^{k(\sigma)}$, and $t^{-1}gs\in S^{2k(\sigma)+1}$ for $s,t\in S^{k(\sigma)}$, we get using the assumption:
$$|\langle\sigma(g)e_i,e_j\rangle-\langle\pi(g)\eta_i,\eta_j\rangle|\leq \varepsilon\sum_{s,t\in G}|f_i(s)||f_j(t)|=\varepsilon\|f_i\|_1\|f_j\|_1\leq\varepsilon M^2.$$
\end{proof}

In the previous proof, by applying the Gram-Schmidt orthonormalization process to the $\eta_i$'s, it is possible to show that the $d\times d$-submatrix $(\langle\sigma(\cdot)e_i,e_j\rangle)_{1\leq i,j\leq d}$ of $\sigma$, is close on $S$ to some $d\times d$-submatrix of $\alpha$, with an explicit bound; but we don't need it at this point.

\subsection{Quantifying Wang's theorem}

Let $\mathcal{H}_\sigma$ be the ($d$-dimensional) Hilbert space of $\sigma$, and let $\mathcal{H}_{\overline{\sigma}}$ be the conjugate Hilbert space (with complex conjugate scalar multiplication and complex conjugate inner product), equipped with the conjugate representation $\overline{\sigma}$. Form the tensor product $\mathcal{H}_{\overline{\sigma}}\otimes\mathcal{H}_\pi$, carrying the representation $\overline{\sigma}\otimes\pi$. Set $\xi_i=e_i\otimes\eta_i$ and $\xi=\sum_{i=1}^d \xi_i\in \mathcal{H}_{\overline{\sigma}}\otimes\mathcal{H}_\pi$, where the $e_i$'s and $\eta_i$'s are as in the section above; observe that the $\xi_i$'s are pairwise orthogonal. We need an estimate on how $\xi$ is moved by $\overline{\sigma}\otimes\pi$.

$$\|\xi-(\overline{\sigma}\otimes\pi)(g)\xi\|^2=2\|\xi\|^2-2\re\langle(\overline{\sigma}\otimes\pi)(g)\xi,\xi\rangle$$
$$=2\sum_{i=1}^d\|\xi_i\|^2-2\sum_{i,j=1}^d \re\langle(\overline{\sigma}\otimes\pi)(g)\xi_i,\xi_j\rangle$$
$$=2\sum_{i=1}^d\|\eta_i\|^2-2\sum_{i,j=1}^d \re\langle e_j,\sigma(g)e_i\rangle\langle\pi(g)\eta_i,\eta_j\rangle.$$
Observe that for every $g\in G$ we have: $d=\sum_{i,j=1}^d\langle\sigma(g)e_i,e_j\rangle\langle e_j,\sigma(g)e_i\rangle$ as the $e_i$'s are an orthonormal basis. Subtracting and adding $2d$ to the previous formula we get:
$$\|\xi-(\overline{\sigma}\otimes\pi)(g)\xi\|^2=2[\sum_{i=1}^d(\|\eta_i\|^2-1)]- 2\sum_{i,j=1}^d\re\langle e_j,\sigma(g)e_i\rangle(\langle\pi(g)\eta_i,\eta_j\rangle-\langle\sigma(g)e_i,e_j\rangle)$$
hence, using Cauchy-Schwarz:
\begin{equation}\label{mainineq}
\|\xi-(\overline{\sigma}\otimes\pi)(g)\xi\|^2 \leq 2\sum_{i=1}^d|\|\eta_i\|^2-1| + 2\sum_{i,j=1}^d|\langle\pi(g)\eta_i,\eta_j\rangle-\langle\sigma(g)e_i,e_j\rangle|
\end{equation}

Theorem \ref{mainresult} will follow immediately form the next Proposition, together with lemma \ref{burnside}

\begin{prop}\label{le:Prop119}
Let $G$ be a discrete Kazhdan group, $S$ a finite symmetric generating set with $e\in S$, and let $\phi$ be a normalized positive definite function on $G$ associated with a finite-dimensional unitary irreducible representation $\sigma$ of $G$. For every $0<\delta<1$ there exists $\varepsilon_{\phi, \delta}>0$ such that for every $\pi \in \Rep(G,\Hil)$, and every unit vector $x \in \Hil$ such that $|\phi(s)-\phi_{\pi,x}(s)|<\varepsilon_{\phi, \delta}$ for $s\in S^{2k(\sigma)+1}$, there exists a unit vector $x' \in \Hil$ with $\|  x-x' \|\leq\delta$ such that the restriction of $\pi$ to the span of $\pi(G)x'$ is isomorphic to $\sigma$.
\end{prop}

\begin{proof} Set $d=\dim\sigma$, let $v$ be a unit vector in $\Hil_\sigma$ such that $\phi(g)=\langle\sigma(g)v,v\rangle$ for every $g\in G$. As in section \ref{weakcon}, for an orthonormal basis $e_1,...,e_d$ of $\mathbb{C}^d$, we find functions $f_1,...,f_d\in\mathbb{C}S^{k(\sigma)}$, with $\max_i \|f_i\|_1\leq M$, such that $\sigma(f_i)v=e_i\,(i=1,...,d)$. 

Let $0<\varepsilon_0< 2$ be such that $(S,\varepsilon_0)$ is a Kazhdan pair for $G$. Fix $\delta$ with $0<\delta<1$, and set 
$$\varepsilon_{\phi,\delta}=\varepsilon=\frac{\delta^2\varepsilon_0^2}{24d(d+1)M^2}.$$
 Let $\pi\in\Rep(G,\Hil)$ and $x\in\Hil$ be a unit vector with $|\phi_{\pi,x}(s)-\phi(s)|<\varepsilon$ for $s\in S^{2k(\sigma)+1}$. Set $\eta_i=\pi(f_i)x$. We may assume that $e_1=v$ and the function $f_1$ is $\delta_e$, so that $\eta_1=x$. We want to prove that the vector $\xi = \sum_{i=1}^d (e_i\otimes \eta_i) \in \Hil_{\overline{\sigma}}\otimes\Hil$ is $(S,t\varepsilon_0)$-invariant for some $0<t<1$, in order to apply Proposition \ref{PropBHV}. 

For $g\in S$ we have, by lemma \ref{Zimmer} and the inequality \ref{mainineq}:
$$\|\xi-(\overline{\sigma}\otimes\pi)(g)\xi\|^2\leq 2d\varepsilon M^2+2d^2\varepsilon M^2=2d(d+1)\varepsilon M^2=\frac{\delta^2\varepsilon_0^2}{12}$$
Again by lemma \ref{Zimmer}, evaluated at $g=e$, we have: $|\|\eta_i\|^2-1|\leq\varepsilon M^2<\frac{1}{2}$, hence $\frac{1}{2}\leq \|\eta_i\|^2\leq \frac{3}{2}$ and $\frac{d}{2}\leq\|\xi\|^2=\sum_{i=1}^d \|\eta_i\|^2\leq\frac{3d}{2}$.
So that, for $g\in S$:
$$\|\xi-(\overline{\sigma}\otimes\pi)(g)\xi\|^2\leq \frac{\delta^2\varepsilon_0^2}{6d}\|\xi\|^2.$$
By Proposition \ref{PropBHV}, there exists a $G$-fixed $\xi'\in\Hil_{\overline{\sigma}}\otimes\Hil$ such that $\|\xi-\xi'\|^2\leq\frac{\delta^2}{6d}\|\xi\|^2$.
Write $\xi'=\sum_{i=1}^d e_i\otimes\zeta_i$, so that  $\|\xi-\xi'\|^2=\sum_{i=1}^d\|\eta_i-\zeta_i\|^2$.
Identify $\mathcal{H}_{\overline{\sigma}}\otimes\mathcal{H}$ with the space of linear operators from $\mathcal{H}_\sigma$ to $\mathcal{H}$ (endowed with the Hilbert-Schmidt norm), via $u\otimes y\mapsto (w\mapsto \langle w,u\rangle y$). Then $\xi'$ identifies with the operator $w\mapsto\sum_{i=1}^d \langle w,e_i\rangle\zeta_i$, which is therefore an intertwining operator between $\sigma$ and $\pi$. The image of this operator, which is $span\{\zeta_1,...,\zeta_d\}$, carries a sub-representation of $\pi$ unitarily equivalent to $\sigma$ (by Schur's lemma). Set $x''=\zeta_1$, then:
$$\|x-x''\|^2=\|\eta_1-\zeta_1\|^2\leq \sum_{i=1}^d\|\eta_i-\zeta_i\|^2=\|\xi-\xi'\|^2\leq\frac{\delta^2}{6d}\|\xi\|^2\leq\frac{\delta^2}{6d}\frac{3d}{2}=\frac{\delta^2}{4},$$
i.e. $\|x-x''\|\leq\frac{\delta}{2}$. Finally, set $x'=\frac{x''}{\|x''\|}$, a unit vector in $\Hil$. Then by the triangle inequality:
$$\|x-x'\|\leq\|x-x''\|+\|x''-x'\|=\|x-x''\| + \|x''\||1-\frac{1}{\|x''\|}|$$
$$=\|x-x''\|+|\|x''\|-\|x\||\leq 2\|x-x''\|\leq\delta.$$
This concludes the proof.
\end{proof}

\begin{question} In the previous proof, the constant $\varepsilon_{\phi,\delta}$ depends on $\sigma$ through the dimension $d$ and the constant $M$ from lemma \ref{compact}. By Theorem 2.6 in \cite{Wa}, a discrete Kazhdan group has finitely many unitary irreducible representations of a given finite dimension (up to unitary equivalence), so Theorem \ref{mainresult} can be made uniform over all unitary irreducible representations $\sigma$ with dimension less than a given dimension. Can it be made uniform over {\it all} finite-dimensional unitary representations?
\end{question}

\section{Proof of Theorem \ref{thm:generic_rep}}

Let $\{U_n\}$ be a countable basis of open sets in the unit sphere $K$ of $\Hil$, and let $\Phi$ be the set of all positive definite functions on $G$ defining irreducible finite dimensional representations. Notice that the set $X' \subseteq \Rep(G,\Hil)$ of all representations $\pi$ such that for every $n \in \NN$ and every $\delta>0$ there exist $m >0$, $x \in U_n$, $x_i \in K$, $c_i \in \CC \setminus \{0\}$, and $\phi_i \in \Phi$, $i \leq m$, such that the $x_i$'s are pairwise orthogonal, $x=\sum c_i x_i$, and each $x_i$ is $(\pi, \phi_i,\varepsilon_{\phi_i, \delta'_i})$-good, where $\delta'_i=\frac{\delta}{|c_i|. m}$, and $\varepsilon_{\phi_i, \delta'_i}$ is given by Theorem \ref{mainresult}, is a $G_\delta$ set. Indeed, for fixed $n$, $\delta$, $m$, $x$, $\bar{x}=(x_1, \ldots, x_m)$, $\bar{c}=(c_1,\ldots, c_m)$, $\bar{\phi}=(\phi_1, \ldots, \phi_m)$ as above, the set
\[ V^{n,\delta,m}_{x, \bar{x},\bar{c},\bar{\phi}}= \{ \pi \in \Rep(G,\Hil): \mbox{ each } x_i \mbox{ is }
 (\pi, \phi_i,\varepsilon_{\phi_i, \delta'_i})\mbox{-good} \} \]
is clearly open. We also put $V^{n,\delta,m}_{x, \bar{x},\bar{c},\bar{\phi}}$ to be the empty set if the $x_i$'s are not pairwise orthogonal or $x \neq \sum c_i x_i$. Now we can define $X'$ by
\[ X'=\bigcap_{n \in \NN} \bigcap_{\delta \in \QQ^{+}} \bigcup_{m \in \NN} \bigcup_{x \in U_n} \bigcup_{\bar{x} \in K^m} \bigcup_{\bar{c} \in \CC^m} \bigcup_{\bar{\phi} \in \Phi^m} V^{n,\epsilon,m}_{x, \bar{x},\bar{c},\bar{\phi}}, \]
which is a $G_\delta$ condition.

 Moreover, $X'$ is dense in $\Rep(G,\Hil)$ as it contains all direct sums of finite-dimensional representations, which, by our assumption, are dense in $\Rep(G, \Hil)$. This is because it is easy to see that for every such sum $\pi$ there are densely many elements $x \in K$ of the form $\sum c_i x_i$, where $x_i$ are pairwise orthogonal unit vectors, and each $x_i$ is $(\pi,\phi_i,\delta)$-good for some $\phi_i$ and every $\delta>0$.
	
Now we show that every representation in $X'$ is a direct sum of finite-dimensional representations. Fix $\pi \in X'$. Using Zorn's lemma, we can decompose $\Hil$ into $\Hil_0$ and $\Hil_1$ such that $\Hil_0$ is the direct sum of all finite-dimensional representations contained in $\pi$. For $i=0,1$, let $P_{\Hil_i}$ be the orthogonal projection of $\Hil$ on $\Hil_i$. Suppose that $\Hil_1$ is not trivial, and fix $0<\delta<1$, $x \in K$, pairwise orthogonal $x_i \in K$ and $c_i \in \CC \setminus \{0\}$, $i \leq m$, such that $x=\sum c_i x_i$, each $x_i$ is $(\pi, \phi_i,\varepsilon_{\phi_i,\frac{\delta}{|c_i|. m}})$-good for some $\phi_i \in \Phi$, and $\left\| x-P_{\Hil_0}x \right\|>\delta$ (the last condition can be satisfied by choosing $x$ in an appropriate $U_n$.) By Theorem \ref{mainresult}, there exist $x'_i \in K$, $i \leq m$, inducing irreducible finite-dimensional representations, and such that $\left\| x_i-x'_i \right\|<\frac{\delta}{|c_i|.m}$, that is, $\left\| x -\sum c_i x'_i \right\|< \delta$. But then, clearly, $x'_{i_0} \not \in \Hil_0$ for some $i_0 \leq m$, as if it was not the case, we would get that $\left\| x- \sum c_i x'_i \right\| \geq \left\| x-P_{\Hil_0}x \right\| > \delta$. Since $P_{\Hil_1}$ is a $G$-intertwiner, the image under $P_{\Hil_1}$ of the linear span of $\pi(G)x'_{i_0}$, is an invariant subspace of $\Hil_1$, which is a contradiction.
	
Now let $X''$ be the set of all those representations that contain every finite dimensional representation with infinite multiplicity. As $G$ is a Kazhdan group, we can see that $X''$ is given by a $G_\delta$ condition. Indeed, for $[\sigma]$ the isomorphism class of a finite-dimensional unitary irreducible representation of $G$, and $n>0$, let $V_{[\sigma],n}$ be the set of representations $\pi\in\Rep(G,\Hil)$ such that $[\sigma]$ appears in $\pi$ with multiplicity at least $n$. Clearly $V_{[\sigma],n}$ is open and
$$X''=\bigcap_{[\sigma]}\bigcap_n V_{[\sigma],n},$$
where the intersection is countable because there are countably many $[\sigma]$'s.

By our assumption on $C^*(G)$, the set $X''$ is dense. Thus, $X=X' \cap X''$ is a dense $G_\delta$ set, all the representations of which are direct sums of finite dimensional representations, each appearing with infinite multiplicity. Clearly, all elements in $X$ are conjugate.
\hfill $\qed$

\begin{remark}
The converse of Theorem \ref{thm:generic_rep} also follows from Theorem 2.5 in \cite{KLP}. That is, if either $G$ does not have property (T), or $C^*(G)$ is not residually finite-dimensional, then all classes in $\Rep(G,\Hil)$ are meager. Indeed, Theorem 2.5 from \cite{KLP} says: if for a separable C*-algebra $A$ the set of isolated points in $\hat A$ is not dense, then the restriction of the action of $U(\Hil)$ by conjugation on a dense $G_\delta$ invariant subset of $\Rep(A,\Hil)$ is turbulent. That, by the definition of turbulence, in particular implies that every class in $\Rep(G,\Hil)$ is meager. Now take $A=C^*(G)$: as isolated points in $\hat G$ correspond to finite-dimensional representations, it follows that when $G$ does not have property (T), $\hat{G}$ does not have isolated points, by Theorem 2.1 in Wang \cite{Wa}; when $C^*(G)$ is not residually finite-dimensional, then the isolated points in $\hat{G}$ are not dense by Archbold's main result in \cite{Ar}.
\end{remark}

\section{Proof of Theorem \ref{kechris}}

For a unitary representation $\pi$, we denote by $\infty\cdot\pi$ the $\ell^2$-direct sum of countably many copies of $\pi$.

\begin{lem}\label{koopman} Let $H$ be a locally compact group. Assume that $H$ has (up to unitary equivalence) countably many finite-dimensional irreducible unitary representations $\sigma_1,\sigma_2,...$. Then the representation $\oplus_{n=1}^\infty \infty\cdot\sigma_n$ is unitarily equivalent to a Koopman representation. 
\end{lem}

\begin{proof} View $\sigma_n$ as a continuous homomorphism $H\rightarrow U(N_n)$. Let $K_n$ denote the closure of $\sigma_n(H)$ in $U(N_n)$, so that $K_n$ is a compact group (on which $H$ acts by left translations by elements of $\sigma_n(H)$). Let $m_n$ denote normalized Haar measure on $K_n$, and let $\lambda_n$ denote the regular representation of $K_n$ on $L^2(K_n,m_m)$. For $p\geq 1$, let $K_{n,p}$ denote a copy of $K_n$ endowed with the measure $2^{-n-p}m_n$. Set $X=\coprod_{n,p}K_{n,p}$, endowed with the $H$-invariant probability measure $\mu=\oplus_{n,p}2^{-n-p}m_n$. Note that the $H$-representations on $L^2(X,\mu)$ and on $L^2_0(X,\mu)$ are equivalent, as $L^2(X,\mu)$ contains the trivial representation with infinite multiplicity.

So it is enough to prove that the $H$-representation $\pi$ on $L^2(X,\mu)$ is equivalent to $\oplus_{n=1}^\infty \infty\cdot\sigma_n$. To see this, first observe that $\pi$ is equivalent to $\oplus_n \infty\cdot\pi_n$, where $\pi_n=\lambda_n\circ\sigma_n$. By Peter-Weyl, $\pi_n$ decomposes as a direct sum of finite-dimensional irreducible representations of $H$, hence of certain $\sigma_k$'s, and moreover $\sigma_n$ is a sub-representation of $\pi_n$ (because the natural representation of $K_n$ on $\mathbb{C}^{N_n}$ is irreducible, hence appears as a sub-representation of $\lambda_n$). This shows that $\oplus_{n} \infty\cdot\pi_n$ is equivalent to $\oplus_{n} \infty\cdot\sigma_n$.
\end{proof}

To prove Theorem \ref{kechris}, observe that a discrete Kazhdan group $G$ satisfies the assumption of lemma \ref{koopman} (by Theorem 2.6 in \cite{Wa}). Let $(\sigma_n)_{n\in\mathbb{N}}$ be an enumeration of all finite-dimensional irreducible unitary representations of $G$. By Theorem \ref{thm:generic_rep} and its proof, the representation $\bigoplus_{n=1}^\infty \infty\cdot\sigma_n$ has a comeager conjugacy class.

In particular, we get the following statement which was proved in \cite{Dol} only for finite abelian groups.
\begin{cor}
Let $G$ be a finite group. Then the set of unitary representations realizable by an action is comeager in $\Rep(G,\Hil)$.
\end{cor}

\begin{remark} Kechris proves (see section (F) in Appendix H of \cite{Ke2}) that, if $G$ is torsion-free abelian, then the set of representations realizable by an action is meager in $\Rep(G,\Hil)$.
\end{remark}

\end{document}